\documentclass[12pt]{article}%
\usepackage[intlimits]{amsmath}
\usepackage{amssymb}
\usepackage[T1]{fontenc}
\usepackage[sc]{mathpazo}
\usepackage{color}
\usepackage[colorlinks]{hyperref}
\usepackage{amsfonts}
\usepackage{graphicx}%
\definecolor {refcol}{RGB}{40,0,255}
\hypersetup{colorlinks=true,allcolors=refcol}
\makeatletter
\newfont{\footsc}{cmcsc10 at 8truept}
\newfont{\footbf}{cmbx10 at 8truept}
\newfont{\footrm}{cmr10 at 10truept}
\newtheorem{theorem}{Theorem}

\newtheorem{corollary}[theorem]{Corollary}

\newtheorem{definition}[theorem]{Definition}

\newtheorem{proposition}[theorem]{Proposition}
\newtheorem{question}[theorem]{Question}

\newenvironment{proof}[1][Proof]{\noindent{\textbf {#1}  }}  {\hfill$\Box$\bigskip}
\begin{document}

\title{\textbf{Hypergraphs and hypermatrices with symmetric spectrum}}
\author{V. Nikiforov\thanks{Department of Mathematical Sciences, University of
Memphis, Memphis TN 38152, USA. Email: \textit{vnikifrv@memphis.edu}}}
\date{}
\maketitle

\begin{abstract}
It is well known that a graph is bipartite if and only if the spectrum of its
adjacency matrix is symmetric. In the present paper, this assertion is
dissected into three separate matrix results of wider scope, which are
extended also to hypermatrices. To this end the concept of bipartiteness is
generalized by a new monotone property of cubical hypermatrices, called
odd-colorable matrices.\ It is shown that a nonnegative symmetric $r$-matrix
$A$ has a symmetric spectrum if and only if $r$ is even and $A$ is
odd-colorable. This result also solves a problem of Pearson and Zhang about
hypergraphs with symmetric spectrum and disproves a conjecture of Zhou, Sun,
Wang, and Bu.

Separately, similar results are obtained for the $H$-spectram of
hypermatrices.$\medskip$



\end{abstract}

\section{Introduction}

The purpose of this paper is to extend the following well-known result in
spectral graph theory:

\textbf{Theorem B }\emph{A graph is bipartite if and only if its adjacency
matrix has a symmetric spectrum}$.\smallskip$

Recall that the spectrum of a complex square matrix $A$ is called
\emph{symmetric} if it is the same as the spectrum of $-A$.

Notwithstanding the fame of Theorem B, we find that it is a certain mismatch,
obtained by forcing together several more general statements, with no regard
to their key differences. To clarify this point we shall distill a few
one-sided implications from the mix of Theorem B.

Thus, for any $n\times n$ complex matrix $A=\left[  a_{i,j}\right]  $ and
nonempty sets $I\subset\left[  n\right]  $, $J\subset\left[  n\right]  $,
write $A\left[  I,J\right]  $ for the submatrix of all $a_{i,j}$ with $i\in I$
and $j\in J.$ Now, call $A$ \emph{bipartite} if there is a partition $\left[
n\right]  =U\cup W$ such that $A\left[  U,U\right]  =0$ and $A\left[
W,W\right]  =0$. Clearly, the adjacency matrix of a bipartite graph is
bipartite, but the above definition extends to any square matrix.\medskip

As with graphs, by negating eigenvectors over one of the partition sets, we get:

\begin{proposition}
\label{pro1}If a matrix is bipartite, then its spectrum is symmetric.
\end{proposition}

Clearly, Proposition \ref{pro1} immediately implies half of Theorem B, but is
much more general, and besides has nothing to do with graphs. Aiming at the
other half of Theorem B, note that the existence of a general converse of
Proposition \ref{pro1} is highly unlikely. Indeed, the matrix
\[
H_{2}=\left[
\begin{array}
[c]{rr}%
1 & 1\\
1 & -1
\end{array}
\right]
\]
dashes hopes for a converse of Proposition \ref{pro1} even within the class of
real symmetric matrices; additionally, the Kronecker powers of $H_{2}$ provide
infinitely many examples to the same effect.

Furthermore, letting $I_{n}$ be the identity matrix of order $n,$ we see that
the matrices
\[
\left[
\begin{array}
[c]{cc}%
0 & I_{n}\\
I_{n} & 0
\end{array}
\right]  \text{ \ \ \ and \ \ \ }\left[
\begin{array}
[c]{cc}%
I_{n} & 0\\
0 & -I_{n}%
\end{array}
\right]
\]
are cospectral; yet the first one is bipartite, whereas the second one is not.
Therefore, in general, bipartiteness cannot be inferred from the spectra of
real symmetric matrices. Obviously, pairing Proposition \ref{pro1} with a
converse in the spirit of Theorem B will badly squash its scope, so it is
better left as is. Hints for possible development are given in Proposition
\ref{pro1.1} and Question \ref{qu1} below.

Nonetheless, it is interesting to find for which matrices spectral conditions
may imply bipartiteness. With this goal in mind, we narrow the focus to
nonnegative matrices, and arrive at the following statement:

\begin{theorem}
\label{th1}Let $A$ be an irreducible nonnegative square matrix with spectral
radius $\rho.$ If $-\rho$ is an eigenvalue of $A$, then $A$ is bipartite.
\end{theorem}

Theorem \ref{th1} was proved by Cvetkovi\'{c}, Doob, and Sachs (\cite{CDS80},
p. 83) for strongly connected digraphs, but their proof extends with no change
to any irreducible nonnegative square matrix. Let us emphasize that Theorem
\ref{th1} also has nothing to do with graphs---its crux is the irreducibility
of $A,$ and the proof rests entirely on the Perron-Frobenius theory.

It should be noted that in \cite{EsHa81}, Esser and Harary stated similar
results for digraphs, but omitted the requirement for strong\ connectivity,
thereby compromising their theorems. For example, under the spell of Theorem
B, they claim on p. 18 of \cite{EsHa81} that \emph{"A digraph is bipartite if
and only if its adjacency matrix has a symmetric spectrum."} This is easily
disproved; e.g., the digraph with adjacency matrix
\[
A_{1}=\left[
\begin{array}
[c]{cccc}%
0 & 1 & 1 & 1\\
1 & 0 & 1 & 1\\
0 & 0 & 0 & 1\\
0 & 0 & 1 & 0
\end{array}
\right]
\]
has spectrum $\left\{  1,1,-1,-1\right\}  ,$ but is not bipartite. Moreover,
since $A_{1}$ is cospectral with the bipartite matrix%
\[
A_{2}=\left[
\begin{array}
[c]{cccc}%
0 & 1 & 0 & 0\\
1 & 0 & 0 & 0\\
0 & 0 & 0 & 1\\
0 & 0 & 1 & 0
\end{array}
\right]  ,
\]
it follows that bipartite digraphs cannot be characterized by their spectra in
general. For additional examples to the same effect the construction of
$A_{1}$ can be generalized as follows: take two disjoint bipartite graphs $G$
and $H,$ and add all arcs from $G$ to $H.$ The resulting digraph is
non-bipartite, as it contains transitive $3$-cycles, but its spectrum is
symmetric, as it is the multiset union of the spectra of $G$ and $H$.

On the positive side, Theorem \ref{th1} implies the following corollary:

\begin{corollary}
\label{cor1}If $A$ is a symmetric nonnegative matrix with symmetric spectrum,
then $A$ is bipartite.
\end{corollary}

In the light of Proposition \ref{pro1}, Theorem \ref{th1}, and Corollary
\ref{cor1}, we see that Theorem B is just a modest corollary of more general
results. Thus, the main goal of this paper is to generalize Proposition
\ref{pro1}, Theorem \ref{th1}, and Corollary \ref{cor1} to hypergraphs and hypermatrices.

Let us recall that in \cite{PeZh14}, Pearson and Zhang raised a similar
problem, which can be stated as: \emph{characterize all connected }%
$r$\emph{-graphs with symmetric spectrum. }Some incomplete solutions to this
problem were given in \cite{Nik14}, \cite{SSW15}, and \cite{ZSWB14}. In this
paper we find necessary and sufficient conditions for the symmetry of the
spectrum of a nonnegative symmetric hypermatrices and obtain a complete
solution of the problem of Pearson and Zhang. We also disprove a conjecture
stated by Zhou, Sun, Wang, and Bu in \cite{ZSWB14}.

The structure of the remaining part of the paper is as follows: In Section
\ref{HM} we give some basic definitions and results for hypermatrices and
hypergraphs. Section \ref{OC} is dedicated to properties of hypermatrices and
hypergraphs that generalize bipartiteness. \ In that section we construct some
families of hypergraphs that disprove the conjecture of Zhou et al. mentioned
above. The main results of the paper are in Section \ref{OE}, where we
characterize nonnegative symmetric hypermatrices with symmetric spectrum.
Finally, in Section \ref{CR} we discuss a few general questions of spectral
hypergraph theory.

\section{\label{HM}Hypermatrices and their eigenvalues}

Let $r\geq2,$ and let $n_{1},\ldots,n_{r}$ be positive integers. An
$r$\emph{-matrix} of order $n_{1}\times\cdots\times n_{r}$ is a function
defined on the Cartesian product $\left[  n_{1}\right]  \times\cdots
\times\left[  n_{r}\right]  .$ In this note we consider only the case
$n_{1}=n,\ldots,n_{r}=n$ and call such an $r$-matrix a \emph{cubical}
$r$-\emph{matrix} of order $n.$\footnote{In graph theory the order of a
(hyper)graph is the number of its vertices, and in much of matrix theory the
order of a square matrix means the number of its rows. We keep these
meanings.}

Hereafter, \textquotedblleft matrix\textquotedblright\ will stand for
\textquotedblleft$r$-matrix\textquotedblright\ with unspecified $r$; thus,
ordinary\ matrices will be referred to as \textquotedblleft$2$%
-matrices\textquotedblright. We denote matrices by capital letters, whereas
their values are denoted by the corresponding lowercase letter with the
variables listed as subscripts. For example, if $A$ is an $r$-matrix of order
$n,$ we let $a_{i_{1},\ldots,i_{r}}:=A\left(  i_{1},\ldots,i_{r}\right)  $ for
all $i_{1},\ldots,i_{r}\in\left[  n\right]  .$

In analogy to $2$-matrices, given a cubical $r$-matrix $A$ of order $n$ and a
set $X\subset\left[  n\right]  $, we write $A\left[  X\right]  $ for the
cubical matrix $a_{i_{1},\ldots,i_{r}},$ of all $\left\{  i_{1},\ldots
,i_{r}\right\}  \in X^{r},$ and call $A\left[  X\right]  $ a \emph{principal}
submatrix of $A$ \emph{induced} by $X.$

Let $A$ be a cubical $r$-matrix of order $n.$ Following the general setup of
\cite{Qi05}, define the eigenvalues of $A$ as in \cite{CPZ08}: an
\emph{eigenvalue} of $A$ is a complex number $\lambda$ that satisfies the
equations%
\begin{equation}
\lambda x_{k}^{r-1}=\sum_{i_{2},\ldots,i_{r}}a_{k,i_{2},\ldots,i_{r}}x_{i_{2}%
}\cdots x_{i_{r}}\ \ \ \ k=1,\ldots,n, \label{alg}%
\end{equation}
for some nonzero complex vector $\left(  x_{1},\ldots,x_{n}\right)  $, called
an \emph{eigenvector} to $\lambda.$ The eigenvalues of $A$ are the roots of
its \emph{characteristic polynomial} $\phi_{A}\left(  x\right)  $ (see
\cite{HHLQ13} for details on $\phi_{A}\left(  x\right)  $) and the multiset of
all roots of $\phi_{A}\left(  x\right)  $ is called the \emph{spectrum }\ of
$A.$ In particular, the multiplicity of an eigenvalue $\lambda$ as a root of
$\phi_{A}\left(  x\right)  $ is called the \emph{algebraic multiplicity} of
$\lambda.$ The spectral radius $\rho\left(  A\right)  $ of $A$ is the largest
modulus of its eigenvalues.

A useful subset of eigenvalues was introduced by Qi in \cite{Qi05}: an
$H$\emph{-eigenvalue} of $A$ is a real number $\lambda,$ which satisfies the
equations (\ref{alg}) for some nonzero real vector $\left(  x_{1},\ldots
,x_{n}\right)  $, called an $H$-$\emph{eigenvector}$ to $\lambda.$ The
$H$\emph{-spectrum} of $A$ is defined as the set of all $H$-eigenvalues and
the $H$\emph{-spectral radius} $\rho_{H}\left(  A\right)  $ of $A$ is the
largest modulus of its $H$-eigenvalues.

Defining symmetry of the spectrum of a cubical hypermatrix is a problem of its
own. Following the familiar path, we say that the spectrum of a cubical matrix
$A$ is \emph{symmetric} if it is the same as the spectrum of $-A.$ It can be
shown that the spectrum of $A$ is symmetric if and only if for every
eigenvalue $\lambda$ of $A$, $-\lambda$ is also an eigenvalue of $A$ with the
same algebraic multiplicity as $\lambda$. In contrast, we say that the
$H$-spectrum of a cubical matrix $A$ is \emph{symmetric} if for every
$H$-eigenvalue $\lambda$ of $A$, $-\lambda$ is also an $H$-eigenvalue of $A$.
Thus, eigenvalue multiplicity is irrelevant for the symmetry of the $H$-spectrum.

We also suggest a geometric spectral symmetry:

\begin{definition}
The spectrum of a cubical matrix $A$ of order $n$ is called
\textbf{geosymmetric} if there exists a unitary diagonal operator
$\Phi:\mathbb{C}^{n}\rightarrow\mathbb{C}^{n}$ such that if $\lambda$ is an
eigenvalue of $A$ with eigenvector $\mathbf{x},$ then $-\lambda$ is also an
eigenvalue of $A$ with eigenvector $\Phi\left(  \mathbf{x}\right)  $. \ 
\end{definition}

\begin{definition}
The $H$-spectrum of a cubical matrix $A$ of order $n$ is called
\textbf{geosymmetric} if there exists an orthogonal diagonal operator
$\Phi:\mathbb{R}^{n}\rightarrow\mathbb{R}^{n}$ such that if $\lambda$ is an
eigenvalue of $A$ with eigenvector $\mathbf{x},$ then $-\lambda$ is also an
eigenvalue of $A$ with eigenvector $\Phi\left(  \mathbf{x}\right)  $.
\end{definition}

Geosymmetric spectrum is a new concept even for $2$-matrices. For example, the
spectrum of the $2$-matrix%
\[
\left[
\begin{array}
[c]{cc}%
I_{n} & 0\\
0 & -I_{n}%
\end{array}
\right]
\]
is symmetric. but not geosymmetric. Hence, Proposition \ref{pro1} may be put
in a stronger form:

\begin{proposition}
\label{pro1.1}If a $2$-matrix is bipartite, then its spectrum is geosymmetric.
\end{proposition}

The new detail prompts a new search for a reasonable converse:

\begin{question}
\label{qu1}Which square $2$-matrices with geosymmetric spectrum are bipartite?
\end{question}

It is not hard to see that for $2$-matrices \textquotedblleft geosymmetric
spectrum\textquotedblright\ implies \textquotedblleft symmetric
spectrum\textquotedblright,\ and for any matrix \textquotedblleft geosymmetric
$H$-spectrum\textquotedblright\ implies \textquotedblleft symmetric
$H$-spectrum\textquotedblright, but the general relation is not so clear:

\begin{question}
Let $r\geq3$ and $A$ be an $r$-matrix with geosymmetric spectrum. Is it always
true that the spectrum of $A$\ is symmetric?
\end{question}

We shall also need a Perron-Frobenius type theorem, so we give some
definitions next: The \emph{digraph} $\mathcal{D}\left(  A\right)  $ of a
cubical $r$-matrix of order $n$ is defined by setting $V\left(  \mathcal{D}%
\left(  A\right)  \right)  :=\left[  n\right]  $ and letting $\left\{
k,j\right\}  \in E\left(  \mathcal{D}\left(  A\right)  \right)  $ whenever
there is a nonzero entry $a_{k,i_{2},\ldots,i_{r}}$ such that $j\in\left\{
i_{2},\ldots,i_{r}\right\}  .$ Following \cite{FGH11}, a cubical matrix is
called \emph{weakly irreducible} if its digraph is strongly connected; if a
cubical matrix is not weakly irreducible, it is called \emph{weakly reducible}.

The combined work of Chang, Pearson, and Zhang \cite{CPZ08}, Yang and Yang
\cite{YaYa10}, and Friedland, Gaubert, and Han \cite{FGH11} laid the ground
for a Perron-Frobenius theory of nonnegative hypermatrices. Of this large body
of work we shall need the following theorem:

\begin{theorem}
\label{PF}If $A$ is a nonnegative cubical matrix, then $\rho\left(  A\right)
$ is an eigenvalue of $A.$ If $A$ is also weakly irreducible and $\mathbf{x}$
is a nonnegative eigenvector to $\rho\left(  A\right)  $, then $\mathbf{x}$ is positive.
\end{theorem}

\subsection{Eigenvalues of real symmetric matrices}

A cubical $r$-matrix is called \emph{symmetric} if $a_{i_{1},\ldots,i_{r}%
}=a_{p\left(  i_{1},\ldots,i_{r}\right)  }$ for every $\left(  i_{1}%
,\ldots,i_{r}\right)  \in\left[  n\right]  ^{r}$ and every permutation
$p\left(  i_{1},\ldots,i_{r}\right)  $ of $\left(  i_{1},\ldots,i_{r}\right)
.$

For a real symmetric $2$-matrix eigenvalues can be alternatively defined by
taking the Lagrange multipliers at the critical points of the matrix quadratic
form over the Euclidean sphere. In \cite{Qi05} and \cite{Qi13}, Qi showed that
some of these relations carry over to $r$-matrices.

Let $A$ be a real symmetric $r$-matrix of order $n,$ and for any real vector
$\mathbf{x}:=\left(  x_{1},\ldots,x_{n}\right)  $, define the \emph{polynomial
form} $P_{A}\left(  \mathbf{x}\right)  $ of $A$ as
\[
P_{A}\left(  \mathbf{x}\right)  :=\sum_{i_{1},\ldots,i_{r}}a_{i_{1}%
,\ldots,i_{r}}x_{i_{1}}\cdots x_{i_{r}}.
\]
Note that polynomial forms generalize quadratic forms to $r$-matrices. In
particular, note the crucial identity%
\[
\frac{dP_{A}\left(  \mathbf{x}\right)  }{dx_{k}}=r\sum_{i_{2},\ldots,i_{r}%
}a_{k,i_{2},\ldots,i_{r}}x_{i_{2}}\cdots x_{i_{r}}.
\]
Further, write $\mathbb{S}_{r}^{n-1}$ for the set of all real $n$-vectors
$\left(  x_{1},\ldots,x_{n}\right)  $ with $\left\vert x_{1}\right\vert
^{r}+\cdots+\left\vert x_{n}\right\vert ^{r}=1,$ and define the parameter
$\eta\left(  A\right)  $ of $A$ as
\[
\eta\left(  A\right)  :=\max\{P_{A}\left(  \mathbf{x}\right)  :\mathbf{x}%
\in\mathbb{S}_{r}^{n-1}\}.
\]
In \cite{Qi13}, Qi showed that if $A$ is a symmetric nonnegative matrix, then
$\rho\left(  A\right)  =\eta\left(  A\right)  $\footnote{For hypergraphs this
equality has been proved by Cooper and Dutle in \cite{CoDu11}.}. We shall need
this result with an extra detail, so we reproduce the proof of Qi.

\begin{proposition}
\label{proC}If $A$ is a nonnegative symmetric matrix, then $\rho\left(
A\right)  =\eta\left(  A\right)  .$ If $\mathbf{x}\in\mathbb{S}_{r}^{n-1}$ and
$\eta\left(  A\right)  =P_{A}\left(  \mathbf{x}\right)  $, then $\mathbf{x}$
is an $H$-eigenvector to $\rho\left(  A\right)  $.
\end{proposition}

\begin{proof}
Suppose that $\left(  x_{1},\ldots,x_{n}\right)  $ is a nonnegative
eigenvector to $\rho\left(  A\right)  $. Since equations (\ref{alg}) are
homogeneous, we may force $x_{1}^{r}+\cdots+x_{n}^{r}=1,$ and so
\[
\rho\left(  A\right)  =\rho\left(  A\right)  \sum_{k\in\left[  n\right]
}x_{k}^{r}=\sum_{k\in\left[  n\right]  }\sum_{i_{2},\ldots,i_{r}}%
a_{k,i_{2},\ldots,i_{r}}x_{k}x_{i_{2}}\cdots x_{i_{r}}=P_{A}\left(
\mathbf{x}\right)  \leq\eta\left(  A\right)  .
\]

Next, let $\eta\left(  A\right)  =$ $P_{A}\left(  x_{1},\ldots,x_{n}\right)  $
for some $\left(  x_{1},\ldots,x_{n}\right)  \in\mathbb{S}_{r}^{n-1}$. \ The
function $\left\vert y_{1}\right\vert ^{r}+\ldots+\left\vert y_{n}\right\vert
^{r}$ has continuous derivatives in each variable; hence the Lagrange
multiplier method implies that there is some $\lambda$ such that
\[
\lambda x_{k}^{r-1}=\sum_{i_{2},\ldots,i_{r}}a_{k,i_{2},\ldots,i_{r}}x_{i_{2}%
}\cdots x_{i_{r}},\text{ \ \ \ \ }k=1,\ldots,n.
\]
Thus, $\lambda$ is an $H$-eigenvalue of $A.$ Now we find that
\[
\rho\left(  A\right)  \geq\lambda\geq\lambda\sum_{k\in\left[  n\right]  }%
x_{k}^{r}=\sum_{k\in\left[  n\right]  }\sum_{i_{2},\ldots,i_{r}}%
a_{k,i_{2},\ldots,i_{r}}x_{k}x_{i_{2}}\cdots x_{i_{r}}=\eta\left(  A\right)
\geq\rho\left(  A\right)  .
\]
Hence, equality holds throughout, and so $\left(  x_{1},\ldots,x_{n}\right)  $
is an eigenvector to $\rho\left(  A\right)  $.
\end{proof}

Note that the digraph $\mathcal{D}\left(  A\right)  $ of a symmetric matrix
$A$ is an undirected $2$-graph. If $A$ is a weakly reducible symmetric matrix,
then $\mathcal{D}\left(  A\right)  $ is disconnected and the vertices of each
component of $\mathcal{D}\left(  A\right)  $ induce a weakly irreducible
principal submatrix of $A,$ called a \emph{component }of $A$. Clearly, $A$ is
a\ block diagonal matrix of its components. It is not hard to see that every
eigenvalue of $A$ is an eigenvalue of one or more of its components, and vice
versa. In \cite{CoDu11}, Cooper and Dutle proved a result about the
characteristic polynomial of disconnected graphs, which was extended by Hu,
Huang, Ling, and Qi \cite{HHLQ13} to weakly reducible matrices, and in
\cite{SSZ13}, Shao, Shan, and Zhang deduced an explicit relation between the
characteristic polynomials of an $r$-matrix $A$ and of its components:

\emph{Let }$A$\emph{ be a symmetric weakly reducible }$r$\emph{-matrix of
order }$n$\emph{. If }$A_{1},\ldots,A_{k}$\emph{ are the components of }%
$A$\emph{ and }$n_{1},\ldots,n_{k}$\emph{ are their orders, then}
\begin{equation}
\phi_{A}\left(  x\right)  =\prod\limits_{i\in\left[  k\right]  }\left(
\phi_{A_{i}}\left(  x\right)  \right)  ^{\left(  r-1\right)  ^{n-n_{i}}}.
\label{det}%
\end{equation}

\subsection{Hypergraphs}

An $r$-\emph{graph} consists of a set of \emph{vertices} $V\left(  G\right)  $
and a set of \emph{edges} $E\left(  G\right)  $, which are subsets of
$V\left(  G\right)  $ with exactly $r$ elements. Hereafter, \textquotedblleft
graph\textquotedblright\ will stand for \textquotedblleft$r$%
-graph\textquotedblright\ with unspecified $r$; thus, ordinary\ graphs will be
referred to as \textquotedblleft$2$-graphs\textquotedblright. The \emph{order
}of a graph is the number of its vertices. If $G$ is of order $n$ and
$V\left(  G\right)  $ is not defined explicitly, it is assumed that $V\left(
G\right)  :=[n]$.

Given an $r$-graph $G$ with vertex set $V\left(  G\right)  =\left[  n\right]
$, the \emph{adjacency matrix }$A\left(  G\right)  $ of $G$ is the $r$-matrix
of order $n,$ whose entries are defined by
\begin{equation}
a_{i_{1},\ldots,i_{r}}:=\left\{
\begin{array}
[c]{ll}%
1, & \text{if }\left\{  i_{1},\ldots,i_{r}\right\}  \in E\left(  G\right)
\text{; }\\
0, & \text{otherwise.}%
\end{array}
\right.  \label{AM}%
\end{equation}
The eigenvalues of $G$ are the eigenvalues of $A\left(  G\right)  $, the
$H$-eigenvalues of $G$ are the $H$-eigenvalues of $A\left(  G\right)  $, and
$\rho\left(  G\right)  :=\rho\left(  A\left(  G\right)  \right)  $.

Since the adjacency matrix $A\left(  G\right)  $ of any graph $G$ is
symmetric, the digraph of $A\left(  G\right)  $ is a $2$-graph, which is just
the $2$-section of $G$ (see, e.g., \cite{Nik14}, p. 533). Hence, $A\left(
G\right)  $ is weakly irreducible if and only if $G$ is connected.

A graph $G$ is called $k$\emph{-chromatic}\textbf{ }if its vertices can be
partitioned into $k$ sets so that each edge intersects at least two
sets.\textbf{ }The \emph{chromatic number}\textbf{ }$\chi\left(  G\right)  $
of $G$ is the smallest $k$ for which $G$ is $k$-chromatic. Similarly, a graph
$G$ is called $k$\emph{-partite}\ if its\textbf{ }vertices can be partitioned
into $k$ sets so that no edge has two vertices from the same set.

\section{\label{OC}Odd-colorings and odd transversals}

Let $r\geq2$ and $r$ be even. A cubical $r$-matrix $A$ of order $n$ is called
\emph{odd-colorable} if there exists a map $\varphi:\left[  n\right]
\rightarrow\left[  r\right]  $ such that if $a_{i_{1},\ldots,i_{r}}\neq0$,
then
\[
\varphi\left(  i_{1}\right)  +\cdots+\varphi\left(  i_{r}\right)  =r/2\text{
}(\operatorname{mod}\text{ }r).
\]
The function $\varphi$ is called an \emph{odd-coloring} of $A.$

Accordingly, we say that a graph is \emph{odd-colorable }if its adjacency
matrix is odd-colorable. Note that if a graph is odd colorable, so are its
subgraphs; hence, \textquotedblleft being odd-colorable\textquotedblright\ is
a monotone graph property.

Next, write\ $I_{X}$ for the indicator function of a set $X\subset\left[
n\right]  $, and let $A$ be a cubical $r$-matrix of order $n.$ A set
$X\subset\left[  n\right]  $ is called an \emph{odd transversal} of $A$ if
$a_{i_{1},\ldots,i_{r}}\neq0$ implies that
\[
I_{X}\left(  i_{1}\right)  +\cdots+I_{X}\left(  i_{r}\right)  =1\text{
}(\operatorname{mod}\text{ }2).
\]
A matrix $A$ with an odd transversal is called an \emph{odd transversal
matrix}.\footnote{Odd transversal matrices were introduced by Chen and Qi in
\cite{ChQi15} under the name \textquotedblleft weakly odd-bipartite
tensors\textquotedblright.}

Accordingly, we say that a graph is an odd transversal graph, if its adjacency
matrix has an odd transversal; that is to say, an odd transversal of a graph
is a vertex set that intersects each edge in an odd number of vertices. Note
that \textquotedblleft having an odd transversal\textquotedblright\ also is a
monotone property of graphs.

The purpose of this section is to investigate odd-colorings and odd
transversals of graphs and matrices.\ To begin with, note that if $A$ is a
square $2$-matrix, then the following three properties are equivalent:

- $A$ is odd-colorable;

- $A$ has an odd transversal;

- $A$ is bipartite.

However, for larger $r$ the situation is more complicated. Let us stress the
fact that odd-colorable $r$-matrices are defined only if $r$ is even, whereas
$r$-matrices with odd transversals may exist for any $r$; e.g., the adjacency
matrices of $r$-partite $r$-graphs have odd transversals.

First, we show that if $r$ is even, \textquotedblleft having an odd
transversal\textquotedblright\ always implies \textquotedblleft
odd-colorable.\textquotedblright

\begin{proposition}
\label{pro2}If $r$ is even and $A$ is an $r$-matrix with an odd transversal,
then $A$ is odd-colorable.
\end{proposition}

\begin{proof}
Let $X$ be an odd transversal of $A.$ For every $i\in\left[  n\right]  ,$ let
$\varphi\left(  i\right)  :=\left(  r/2\right)  I_{X}\left(  i\right)  .$ If
$a_{i_{1},\ldots,i_{r}}\neq0,$ then%
\[
I_{X}\left(  i_{1}\right)  +\cdots+I_{X}\left(  i_{r}\right)  =1\text{
}(\operatorname{mod}\text{ }2)\text{;}%
\]
hence,
\[
\varphi\left(  i_{1}\right)  +\cdots+\varphi\left(  i_{r}\right)  =\left(
r/2\right)  I_{X}\left(  i_{1}\right)  +\cdots+\left(  r/2\right)
I_{X}\left(  i_{r}\right)  =r/2\text{ }(\operatorname{mod}\text{ }r).
\]
Therefore, $\varphi\left(  i\right)  $ is an odd-coloring of $A$ and so $A$ is odd-colorable.
\end{proof}

As it turns out, if $r=2$ $(\operatorname{mod}$ $4)$, then Proposition
\ref{pro1} can be inverted, that is to say, \textquotedblleft having an odd
transversal\textquotedblright\ and \textquotedblleft
odd-colorable\textquotedblright\ are equivalent properties if $r=2$
$(\operatorname{mod}$ $4)$.

\begin{proposition}
\label{pro3}Let $r=2$ $(\operatorname{mod}4)$. An $r$-matrix $A$ is
odd-colorable if and only if it has an odd transversal.
\end{proposition}

\begin{proof}
Set $r=4k+2.$ In view of Proposition \ref{pro1}, we only need to show that if
$A$ is odd-colorable, then it has an odd transversal. Let $\varphi:\left[
n\right]  \rightarrow\left[  4k+2\right]  $ be an odd-coloring of $A.$ Write
$X$ for the set of all $i\in\left[  n\right]  $ such that $\varphi\left(
i\right)  $ is odd. We shall show that $X$ is an odd transversal of $A.$
Indeed, if $a_{i_{1},\ldots,i_{r}}\neq0,$ then
\[
\varphi\left(  i_{1}\right)  +\cdots+\varphi\left(  i_{r}\right)  =2k+1\text{
}(\operatorname{mod}\text{ }4k+2).
\]
Therefore, among the numbers $\varphi\left(  i_{1}\right)  ,\ldots
,\varphi\left(  i_{r}\right)  ,$ the number of the odd ones is odd, which
implies that
\[
I_{X}\left(  i_{1}\right)  +\cdots+I_{X}\left(  i_{r}\right)  =1\text{
}(\operatorname{mod}\text{ }2),
\]
and so, $X$ is an odd transversal of $A.$\footnote{Let us note that this
argument has been used before, e.g., in the proof of Theorem 11 of
\cite{ZSWB14}.}
\end{proof}

Since Proposition \ref{pro3}\ does not cover the case $r=0$
$(\operatorname{mod}4)$,\ Zhou, Sun, Wang, and Bu stated a conjecture on p. 9
of \cite{ZSWB14}, which would imply that Proposition \ref{pro3}\ holds for any
even $r$. However, we shall construct two families of odd-colorable graphs
with no odd transversals, thereby disproving this conjecture.

\begin{proposition}
\label{pro4}Let $k$ be a positive integer. If $n\geq8k,$ then there exists a
family of odd-colorable $4k$-graphs of order $n$ with no odd transversals.
\end{proposition}

\begin{proof}
Let $n\geq8k.$ Partition $\left[  n\right]  $ into two sets $A$ and $B$ so
that $\left\vert A\right\vert \geq4k$ and $\left\vert B\right\vert \geq4k.$
Define the $4k$-graph $G$ by setting $V\left(  G\right)  :=\left[  n\right]  $
and letting%
\[
E\left(  G\right)  :=\{e:e\subset\left[  n\right]  ,\text{ }\left\vert e\cap
A\right\vert =2k\text{ and }\left\vert e\cap B\right\vert =2k\}.\text{ }%
\]
To see that $G$ is odd-colorable, define a map $\varphi:\left[  n\right]
\rightarrow\left[  4k\right]  $ by letting
\[
\varphi\left(  i\right)  :=\left\{
\begin{array}
[c]{cc}%
4k, & \text{if }i\in A\text{; }\\
1, & \text{if }i\in B\text{.}%
\end{array}
\right.
\]
For every edge $\left\{  i_{1},\ldots,i_{4k}\right\}  \in E\left(  G\right)
,$ we see that
\[
\varphi\left(  i_{1}\right)  +\cdots+\varphi\left(  i_{4k}\right)  =2k\text{
}(\operatorname{mod}\text{ }4k)\text{;}%
\]
thus, $G$ is odd-colorable.

Assume for a contradiction that $X\subset\left[  n\right]  $ is an odd
transversal. Then either $\left\vert X\cap A\right\vert >2k$ or $\left\vert
X\cap B\right\vert >2k,$ for otherwise there are $e_{1}\subset A\backslash X$
and $e_{2}\subset B\backslash X$ such that $\left\vert e_{1}\right\vert
=\left\vert e_{2}\right\vert =2k$; hence, $e_{1}\cup e_{2}\in E\left(
G\right)  $ and $\left(  e_{1}\cup e_{2}\right)  \cap X=\varnothing,$
contradicting that $X$ is a transversal.

Assume by symmetry that $\left\vert X\cap A\right\vert >2k.$ Then $\left\vert
X\cap B\right\vert <2k,$ for otherwise there are $e_{1}\subset X\cap A$ and
$e_{2}\subset X\cap B$ such that $\left\vert e_{1}\right\vert =\left\vert
e_{2}\right\vert =2k$; hence, $e_{1}\cup e_{2}\in E\left(  G\right)  $ and
$\left(  e_{1}\cup e_{2}\right)  \cap X=4k,$ contradicting that $X$ is an odd transversal.

Since $\left\vert X\cap A\right\vert >2k$ and $\left\vert X\cap B\right\vert
<2k$, there are $e_{1}\subset X\cap A$ and $e_{2}\subset B\backslash X$ such
that $\left\vert e_{1}\right\vert =\left\vert e_{2}\right\vert =2k$; hence,
$e_{1}\cup e_{2}\in E\left(  G\right)  $ and $\left(  e_{1}\cup e_{2}\right)
\cap X=2k,$ contradicting that $X$ is an odd transversal. Therefore, $G$ has
no odd transversals.
\end{proof}

To construct another family of odd-colorable but not odd-transversal graphs,
note that if $r$ is even and $G$ is odd-colorable, then $\chi\left(  G\right)
\leq r,$ whereas if $G$ is odd-transversal, then $\chi\left(  G\right)  =2$.
The graphs constructed in Proposition \ref{pro4} are also $2$-chromatic, but
odd-colorable graphs in general may have higher chromatic number, as shown below:

\begin{proposition}
\label{pro5}Let $k$ be a positive integer. If $n\geq16k,$ then there exists a
family of $3$-chromatic odd-colorable $4k$-graphs of order $n$.
\end{proposition}

\begin{proof}
Let $n\geq16k$ and let partition $\left[  n\right]  $ into three sets $A,$
$B,$ and $C$ so that $\left\vert A\right\vert \geq6k,$ $\left\vert
B\right\vert \geq6k$ and $\left\vert C\right\vert \geq4k.$ First, define four
families of $4k$-subsets of $\left[  n\right]  :$
\begin{align*}
E_{1}  &  :=\{e:e\subset\left[  n\right]  ,\text{ }\left\vert e\cap
A\right\vert =2k\text{ and }\left\vert e\cap C\right\vert =2k\},\\
E_{2}  &  :=\{e:e\subset\left[  n\right]  ,\text{ }\left\vert e\cap
B\right\vert =2k\text{ and }\left\vert e\cap C\right\vert =2k\},\\
E_{3}  &  :=\{e:e\subset\left[  n\right]  ,\text{ }\left\vert e\cap
A\right\vert =k\text{ and }\left\vert e\cap B\right\vert =3k\},\\
E_{4}  &  :=\{e:e\subset\left[  n\right]  ,\text{ }\left\vert e\cap
A\right\vert =3k\text{ and }\left\vert e\cap B\right\vert =k\}.\text{ }%
\end{align*}
Now, define a $4k$-graph $G$ by setting $V\left(  G\right)  :=\left[
n\right]  $ and letting $E\left(  G\right)  :=E_{1}\cup E_{2}\cup E_{3}\cup
E_{4}.$

To see that $G$ is odd-colorable, define the map $\varphi:\left[  n\right]
\rightarrow\left[  4k\right]  $ by letting
\[
\varphi\left(  i\right)  :=\left\{
\begin{array}
[c]{ll}%
1 & \text{if }i\in A\text{; }\\
4k-1 & \text{if }i\in B\text{;}\\
4k & \text{if }i\in C\text{.}%
\end{array}
\right.
\]
Let $e\in$ $E\left(  G\right)  .$ We shall check that $\varphi\left(
i\right)  $ is an odd-coloring. Indeed, if $e\in$ $E_{1},$ then
\[
\sum_{i\in e}\varphi\left(  i\right)  =4k\cdot2k+2k=2k\text{ }%
(\operatorname{mod}4k),
\]
and if $e\in E_{2},$ then
\[
\sum_{i\in e}\varphi\left(  i\right)  =4k\cdot2k+2k\left(  4k-1\right)
=2k\text{ }(\operatorname{mod}4k).
\]
If $e\in E_{3},$ then
\[
\sum_{i\in e}\varphi\left(  i\right)  =k+3k\left(  4k-1\right)  =2k\text{
}(\operatorname{mod}4k),
\]
and finally, if $e\in E_{4},$ then
\[
\sum_{i\in e}\varphi\left(  i\right)  =3k+k\left(  4k-1\right)  =2k\text{
}(\operatorname{mod}4k).
\]
Hence if $\left\{  i_{1},\ldots,i_{4k}\right\}  \in E\left(  G\right)  ,$ then
$\varphi\left(  i_{1}\right)  +\cdots+\varphi\left(  i_{4k}\right)  =2k$
$(\operatorname{mod}$ $4k)$; thus, $G$ is odd-colorable.

Since the classes $A,$ $B,$ and $C$ do not span edges, $\chi\left(  G\right)
\leq3.$ Assume for a contradiction that $\chi\left(  G\right)  =2$, and let
$X$ and $Y$ be the two color classes of $G$. Clearly, either $\left\vert A\cap
X\right\vert \geq3k$ or $\left\vert A\cap Y\right\vert \geq3k.$ By symmetry,
assume\ that $\left\vert A\cap X\right\vert \geq3k.$ Therefore, $\left\vert
Y\cap B\right\vert >5k,$ for otherwise $X$ would contain an edge from $E_{4}$.
Clearly, either $\left\vert C\cap X\right\vert \geq2k$ or $\left\vert C\cap
Y\right\vert \geq2k.$ If $\left\vert A\cap X\right\vert \geq2k,$ then $X$
contains an edge from $E_{1}$; if $\left\vert A\cap Y\right\vert \geq2k,$ then
$Y$ contains an edge from $E_{2}$. This contradiction completes the proof.
\end{proof}

It is not clear how large may the chromatic number of odd-colorable graphs be,
so we would like to raise a question:

\begin{question}
Let $r=0$ $(\operatorname{mod}$ $4).$ What is the maximum chromatic number of
an odd-colorable $r$-graph of order $n.$
\end{question}

\section{\label{OE}Odd-colorings and eigenvalues}

In this section we discuss matrices and graphs with symmetric spectrum. The
first results for $r$-graphs were given in \cite{Nik14}, Theorem 8.9 and
Proposition 8.10 that read as:\medskip

\emph{If }$G$\emph{ is an }$r$\emph{-graph and }$-\rho\left(  G\right)
$\emph{ is an eigenvalue of }$G,$\emph{ then }$r$\emph{ is even}.

\emph{If }$r$\emph{ is even and }$G$\emph{ is an }$r$\emph{-graph with an odd
transversal, then }$G$ \emph{has symmetric spectrum.}\medskip

These facts extend to symmetric nonnegative $r$-matrices as well. Here is our
generalization of Proposition \ref{pro1} from the Introduction.

\begin{theorem}
\label{thOC} Let $r\geq2$ and $r$ be even. If an $r$-matrix $A$ is
odd-colorable, then the spectrum of $A$ is symmetric and geosymmetric.
\end{theorem}

\begin{proof}
Let $n$ be the order of $A$ and $\varphi:\left[  n\right]  \rightarrow\left[
r\right]  $ be an odd coloring of $A,$ i.e., if $a_{j_{1},\ldots,j_{r}}\neq0,$
then
\begin{equation}
\varphi\left(  j_{1}\right)  +\cdots+\varphi\left(  j_{r}\right)  =r/2\text{
}(\operatorname{mod}\text{ }r). \label{fi}%
\end{equation}

First, we show that the spectrum of $A$ is geosymmetric. Define a map
$\Phi:\mathbb{C}^{n}\rightarrow\mathbb{C}^{n}$ by letting $\Phi\left(
x_{1},\ldots,x_{n}\right)  :=\left(  y_{1},\ldots,y_{n}\right)  ,$ where for
each $k\in\left[  n\right]  ,$ we set $y_{k}:=e^{2\varphi\left(  k\right)  \pi
i/r}x_{k}$. Clearly, $\Phi$ is a unitary diagonal operator. Let $\lambda$ be
an eigenvalue of $A$ with eigenvector $\left(  x_{1},\ldots,x_{n}\right)  $.
We see that
\begin{align*}
\sum_{j_{2},\ldots,j_{r}}a_{k,j_{2},\ldots,j_{r}}y_{j_{2}}\cdots y_{j_{r}}  &
=\sum_{j_{2},\ldots,j_{r}}a_{k,j_{2},\ldots,j_{r}}x_{j_{2}}\cdots x_{j_{r}%
}e^{\pi i-2\varphi\left(  k\right)  \pi i/r}=\\
&  =-\sum_{j_{2},\ldots,j_{r}}a_{k,j_{2},\ldots,j_{r}}x_{j_{2}}\cdots
x_{j_{r}}e^{2\varphi\left(  k\right)  \left(  r-1\right)  \pi i/r}\\
&  =-\lambda x_{k}^{r-1}e^{2\varphi\left(  k\right)  \left(  r-1\right)  \pi
i/r}=-\lambda y_{k}^{r-1}.
\end{align*}
Therefore, $-\lambda$ is an eigenvalue of $A$ with eigenvector $\left(
y_{1},\ldots,y_{n}\right)  $, so the spectrum of $A$ is geosymmetric.

To finish the proof, note that in \cite{Sha13}, Shao showed that if $\left(
z_{1},\ldots,z_{n}\right)  $ is a vector with nonzero entries and a matrix $B$
is defined as
\[
b_{j_{1},\ldots,j_{r}}=z_{j_{1}}^{-r}a_{j_{1},\ldots,j_{r}}z_{j_{1}}\cdots
z_{j_{r}},\text{ }\left(  j_{1},\ldots,j_{r}\right)  \in\left[  n\right]
^{r},
\]
then $A$ and $B$ have the same spectrum. Setting $z_{k}:=e^{2\varphi\left(
k\right)  \pi i/r}$ for each $k\in\left[  n\right]  ,$ in view of (\ref{fi}),
we find that $B=-A$; thus the spectrum of $A$ is symmetric, completing the proof.
\end{proof}

Our next goal is to generalize Theorem \ref{th1} to symmetric hypermatrices.
The result that we shall state comes as a consequence of several results of
Yang and Yang \cite{YaYa11}, which for convenience we combine into one
theorem. Recall that in \cite{YaYa11}, Yang and Yang gave numerous results
about weakly irreducible nonnegative matrices with more than one eigenvalue of
modulus equal to the spectral radius. For the case of symmetric matrices,
their Theorems 3.9, 3.10, and 3.11 imply the following statement:

\begin{theorem}
\label{thY}Let $A$ be a weakly irreducible, nonnegative, symmetric $r$-matrix
of order $n.$ If $\rho\left(  A\right)  e^{i\theta}$ is an eigenvalue of $A,$
then there is a function $\varphi:\left[  n\right]  \rightarrow\left[
r\right]  $ such that if $a_{j_{1},\ldots,j_{r}}\neq0,$ then
\[
e^{2\varphi\left(  j_{1}\right)  \pi i/r}\cdots e^{2\varphi\left(
j_{r}\right)  \pi i/r}=e^{i\theta}e^{2\varphi\left(  j_{1}\right)  \pi
i}=\cdots=e^{i\theta}e^{2\varphi\left(  j_{r}\right)  \pi i}.
\]

\end{theorem}

Using this result, we encounter no difficulty in generalizing Theorem
\ref{th1}:

\begin{theorem}
\label{mth}Let $A$ be a weakly irreducible, nonnegative, symmetric $r$-matrix.
If $-\rho\left(  A\right)  $ is an eigenvalue of $A,$ then $r$ is even and $A$
is odd-colorable.
\end{theorem}

\begin{proof}
Let $n$ be the order of $A$. Theorem \ref{thY} implies that there exists a
function $\varphi:\left[  n\right]  \rightarrow\left[  r\right]  $ such that
if $a_{j_{1},\ldots,j_{r}}\neq0,$ then
\[
e^{2\varphi\left(  j_{1}\right)  \pi i/r}\cdots e^{2\varphi\left(
j_{r}\right)  \pi i/r}=e^{i\pi}e^{2\varphi\left(  j_{1}\right)  \pi i}.
\]
Taking the $r$th power of both sides, we find that%
\[
\left(  \varphi\left(  j_{1}\right)  +\cdots+\varphi\left(  j_{r}\right)
\right)  2\pi=r\pi\text{ }(\operatorname{mod}\text{ }2\pi).
\]
Hence, $r$ is even and
\[
\varphi\left(  j_{1}\right)  +\cdots+\varphi\left(  j_{r}\right)  =r/2\text{
}(\operatorname{mod}\text{ }r).
\]
Therefore, $\varphi$ is an odd-coloring of $A,$ and so $A$ is odd-colorable.
\end{proof}

\begin{corollary}
\label{cor2}If $G$ is a connected graph, then the spectrum of $G$ is symmetric
if and only if $r$ is even and $G$ is odd-colorable.
\end{corollary}

Corollary \ref{cor2} completely solves the problem of Pearson and Zhang
mentioned in the introduction. However, it is possible to further strengthen
this assertion by dropping the premise for connectivity. Thus, using Theorem
\ref{mth}, we generalize Corollary \ref{cor1} as follows:

\begin{theorem}
\label{th2}If $A$ is a symmetric nonnegative $r$-matrix with symmetric
spectrum, then $A$ is odd-colorable.
\end{theorem}

\begin{proof}
Assume that $A$ is weakly reducible, for Theorem \ref{mth} takes care of the
other case. Let $A_{1},\ldots,A_{k}$ be the components of $A$ and let
$\rho=\rho\left(  A\right)  .$ Since $\rho$ is an eigenvalue of $A,$ we see
that $-\rho$ also is an eigenvalue of $A$; hence, $-\rho$ is eigenvalue of a
component of $A.$ Without loss of generality we assume that $-\rho$ is an
eigenvalue of $A_{1}.$ Thus, $\rho\left(  A_{1}\right)  =\rho,$ and Theorem
\ref{mth}, together with Theorem \ref{thOC}, implies that the spectrum of
$A_{1}$ is symmetric.

Further, zero all entries of $A_{1}$ and write $A^{\prime}$ for the resulting
$r$-matrix. Using equation (\ref{det}), it is not hard to see that $A^{\prime
}$ also has a symmetric spectrum. Clearly, $A_{2},\ldots,A_{k}$ are components
of $A^{\prime},$ and all other components of $A^{\prime}$ are zero diagonal
entries. Iterating this argument, we end up with a zero matrix, after finding
that all components of $A$ have symmetric spectrum. Hence, all components of
$A$ are odd-colorable, and so is $A$.
\end{proof}

A similar theorem holds also for matrices with geosymmetric spectrum; we omit
its proof.

\begin{theorem}
If $A$ is a symmetric nonnegative $r$-matrix with geosymmetric spectrum, then
$A$ is odd-colorable.
\end{theorem}

Note that the conclusion of Theorem \ref{mth} may fail for non-symmetric
matrices. For instance, let $A$ be the $3$-matrix of order six such that
\[
a_{1,2,3}=a_{2,3,4}=a_{3,4,5}=a_{4,5,6}=a_{5,6,1}=a_{6,1,2}=1,
\]
and all other entries are zero. The six eigenequations of $A$ are
\[
\lambda x_{1}^{2}=x_{2}x_{3},\text{ }\lambda x_{2}^{2}=x_{3}x_{4},\text{
}\lambda x_{3}^{2}=x_{4}x_{5},\text{ }\lambda x_{4}^{2}=x_{5}x_{6},\text{
}\lambda x_{5}^{2}=x_{6}x_{1},\text{ }\lambda x_{6}^{2}=x_{1}x_{2}.
\]
It is not hard to see that if $\lambda\neq0$, then $\lambda^{6}=1,$ and so the
spectral radius of $A$ is $1.$ Define the vector $\left(  y_{1},\ldots
,y_{6}\right)  ,$ by setting $y_{k}:=e^{2k\pi i/6}$ for each $k\in\left\{
1,\ldots,6\right\}  .$ We see that
\begin{align*}
y_{2}y_{3}  &  =e^{4\pi i/6}e^{6\pi i/6}=-y_{1}^{2},\text{ \ }y_{3}%
y_{4}=e^{6\pi i/6}e^{8\pi i/6}=-y_{2}^{2},\text{ \ }y_{4}y_{5}=e^{8\pi
i/6}e^{10\pi i/6}=-y_{3}^{2},\\
y_{5}y_{6}  &  =e^{10\pi i/6}e^{12\pi i/6}=-y_{4}^{2},\text{ \ }y_{6}%
y_{1}=e^{12\pi i/6}e^{2\pi i/6}=-y_{5}^{2},\text{ \ }y_{1}y_{2}=e^{2\pi
i/6}e^{4\pi i/6}=-y_{6}^{2}.
\end{align*}
Hence, $\left(  y_{1},\ldots,y_{6}\right)  $ is an eigenvector to the
eigenvalue $-1.$

Since Theorem \ref{mth} extends Theorem \ref{th1} only to symmetric
$r$-matrices, we conclude with a corresponding question:

\begin{question}
Which weakly irreducible nonnegative cubical $r$-matrices $A$ have
$-\rho\left(  A\right)  $ as an eigenvalue?
\end{question}

\subsection{Odd transversals and $H$-eigenvalues}

The symmetry of the $H$-spectrum of $r$-graphs is somewhat simpler. In
\cite{Nik14}, Theorem 8.7, we proved the following statement:\medskip

\emph{If }$G$\emph{ is a connected graph and }$-\rho\left(  G\right)  $\emph{
is an }$H$\emph{-eigenvalue of }$G,$\emph{ then }$G$\emph{ has an odd
transversal.\medskip}

The short proof of this assertion extends with minor changes to nonnegative matrices:

\begin{theorem}
\label{symH}If $A$ is a weakly irreducible, nonnegative, symmetric $r$-matrix
and $-\rho\left(  A\right)  $ is an $H$-eigenvalue of $A,$ then $r$ is even
and $A$ has an odd transversal.
\end{theorem}

\begin{proof}
Let $A$ be of order $n$ and set $\rho:=\rho\left(  A\right)  .$ Suppose that
$-\rho$ is an $H$-eigenvalue of $A$ and let $\mathbf{x}:=\left(  x_{1}%
,\ldots,x_{n}\right)  $ be an $H$-eigenvector to $-\rho$ such that
$\mathbf{x}\in\mathbb{S}_{r}^{n-1}$ $.$ Clearly, the vector $\mathbf{y}%
:=\left(  \left\vert x_{1}\right\vert ,\ldots,\left\vert x_{n}\right\vert
\right)  $ also belongs to $\mathbb{S}_{r}^{n-1}.$ We have
\[
-\rho x_{k}^{r-1}=\sum_{i_{2},\ldots,i_{r}}a_{k,i_{2},\ldots,i_{r}}x_{i_{2}%
}\cdots x_{i_{r}},\ \ \ \ k=1,\ldots,n.
\]
Hence, for each $k\in\left[  n\right]  ,$ we find that
\begin{equation}
\rho\left\vert x_{k}\right\vert ^{r}=\left\vert \sum_{i_{2},\ldots,i_{r}%
}a_{k,i_{2},\ldots,i_{r}}x_{k}x_{i_{2}}\cdots x_{i_{r}}\right\vert \leq
\sum_{i_{2},\ldots,i_{r}}a_{k,i_{2},\ldots,i_{r}}\left\vert x_{k}\right\vert
|x_{i_{2}}|\cdots\left\vert x_{i_{r}}\right\vert . \label{in2}%
\end{equation}
Adding these inequalities, we get%
\[
\rho=\rho\sum_{k\in\left[  n\right]  }\left\vert x_{k}\right\vert ^{r}\leq
\sum_{k\in\left[  n\right]  }\sum_{i_{2},\ldots,i_{r}}a_{k,i_{2},\ldots,i_{r}%
}\left\vert x_{k}\right\vert |x_{i_{2}}|\cdots\left\vert x_{i_{r}}\right\vert
=P_{G}\left(  \mathbf{y}\right)  \leq\rho.
\]
Therefore $P_{G}\left(  \mathbf{y}\right)  =\rho$ and Proposition \ref{proC}
implies that $\mathbf{y}$ is a nonnegative eigenvector to $\rho,$ which by
Theorem \ref{PF} must be positive. In addition, equality holds in (\ref{in2})
for every $k\in\left[  n\right]  $; thus,
\begin{equation}
-\mathrm{sign}(x_{k}^{r})=\mathrm{sign}(x_{k}x_{i_{2}}\cdots x_{i_{r}})
\label{in3}%
\end{equation}
whenever $a_{k,i_{2},\ldots,i_{r}}\neq0$. Since $A$ is symmetric, we get
\[
\left(  -1\right)  ^{r}\mathrm{sign}(x_{i_{1}}^{r})\cdots\mathrm{sign}%
(x_{i_{r}}^{r})=\mathrm{sign}(x_{i_{1}}\cdots x_{i_{r}})^{r},
\]
whenever $a_{i_{1},\ldots,i_{r}}\neq0$, and thus $r$ is even. Therefore,
(\ref{in3}) implies that $x_{i_{1}}\cdots x_{i_{r}}<0$ whenever $a_{i_{1}%
,\ldots,i_{r}}\neq0,$ and so the set of indices of the negative entries of
$\left(  x_{1},\ldots,x_{n}\right)  $ is an odd transversal of $A.$
\end{proof}

As in Theorem \ref{th2}, we obtain the following generalization of Corollary
\ref{cor1}:

\begin{corollary}
\label{cor3}If $A$ is a symmetric nonnegative $r$-matrix with geosymmetric
$H$-spectrum, then $A$ has an odd-transversal.
\end{corollary}

Finally, here is a converse of Corollary \ref{cor3}, which completes the
picture for the $H$-spectrum.

\begin{theorem}
\label{thOCR} Let $r\geq2$ and $r$ be even. If an $r$-matrix $A$ has an odd
transversal, then its $H$-spectrum is symmetric and geosymmetric.
\end{theorem}

\begin{proof}
Let $A$ be of order $n$ and let $X$ be an odd transversal of $A.$ Hence, if
$a_{j_{1},\ldots,j_{r}}\neq0,$ then
\[
I_{X}\left(  j_{1}\right)  +\cdots+I_{X}\left(  j_{r}\right)  =1\text{
}(\operatorname{mod}\text{ }2).
\]
We shall prove that the $H$-spectrum of $A$ is geosymmetric; the symmetry
follows immediately, as the $H$-spectrum is a simple set.

Define a map $\Phi:\mathbb{R}^{n}\rightarrow\mathbb{R}^{n}$ by letting
$\Phi\left(  x_{1},\ldots,x_{n}\right)  :=\left(  y_{1},\ldots,y_{n}\right)
,$ where for each $k\in\left[  n\right]  ,$ we set $y_{k}:=\left(
2I_{X}\left(  k\right)  -1\right)  x_{k}$. Clearly $\Phi$ is an orthogonal
diagonal operator. Let $\lambda$ be an $H$-eigenvalue of $A$ with
$H$-eigenvector $\left(  x_{1},\ldots,x_{n}\right)  $. Since $2I_{X}\left(
k\right)  -1=-1$ if $I_{X}\left(  k\right)  =1,$ and $2I_{X}\left(  k\right)
-1=1$ if $I_{X}\left(  k\right)  =0,$ we see that
\[
\sum_{j_{2},\ldots,j_{r}}a_{k,j_{2},\ldots,j_{r}}y_{j_{2}}\cdots y_{j_{r}%
}=-\sum_{j_{2},\ldots,j_{r}}a_{k,j_{2},\ldots,j_{r}}x_{j_{2}}\cdots x_{j_{r}%
}=-\lambda y_{k}^{r-1}.
\]
Hence, $-\lambda$ is an eigenvalue with eigenvector $\left(  y_{1}%
,\ldots,y_{n}\right)  $, and so the $H$-spectrum of $A$ is geosymmetric.
\end{proof}

The conclusion of this subsection is that if we considered only real
eigenvectors, the existence of odd-colorable graphs that are not
odd-transversal would not have been made clear. This distinction shows that
complex eigenvectors and eigenvalues may play structural role in spectral
hypergraph theory.

\section{\label{CR}Concluding remarks}

In this short section we briefly address three topics in spectral hypergraph
theory: definition of adjacency matrix, relevance of algebraic spectra, and
\textquotedblleft odd-bipartiteness\textquotedblright.

\textbf{Adjacency matrix. }The traditional definition of the adjacency
(hyper)matrix (see, e.g., \cite{FrWi95} and \cite{KLM13}) represents edges by
$1$. This tradition was challenged by Cooper and Dutle in \cite{CoDu11}, who
chose to represent the edges of an $r$-graph by the value $1/\left(
r-1\right)  !,$ thereby scaling down all eigenvalues and simplifying a number
of expressions. While this novelty has been widely accepted, it has drawbacks.
For example, in the new setup the $k$th slice of the adjacency matrix is not
the adjacency matrix of the link graph of the vertex $k$. We believe that the
correct definition of the adjacency matrix is given by (\ref{AM}), and its
scaled version should be called \textquotedblleft scaled adjacency
matrix\textquotedblright.

\textbf{Relevance of the algebraic spectrum.} Given the adjacency matrix, it
is possible to build a spectral theory for hypergraphs, after adopting a
particular spectral theory of hypermatrices. The theory of the determinants
developed in \cite{HHLQ13} gives a solid ground for such endeavor, but is
hardly acceptable in full generality for hypergraphs. The main problem comes
from the fact that there are overwhelmingly many algebraic eigenvalues:
indeed, a result of Qi in \cite{Qi05} implies that an $r$-graph of order $n$
has $n\left(  r-1\right)  ^{n-1}$ eigenvalues. Certainly not all of those are
combinatorially relevant. For instance, any vector with at most $r-2$ nonzero
entries is an eigenvector to the eigenvalue $0.$

Here is an example showing that algebraic multiplicity may be combinatorially
irrelevant: Let $r\geq3$ and $G$ be an $r$-graph of order $n.$ Let $m_{0}$ is
the multiplicity of the eigenvalue $0$ and $m_{1},\ldots,m_{k}$ be the
multiplicities of the remaining eigenvalues of $G$. Now, add an isolated
vertex to $G$ and write $H$ for the resulting graph. Using (\ref{det}), we see
that $H$ has the same eigenvalues as $G$, but the multiplicity of $0$
increases to $m_{0}^{r-1}+\left(  r-1\right)  ^{n},$ and the multiplicities of
the remaining eigenvalues become $m_{1}^{r-1},\ldots,m_{k}^{r-1}.$ Given how
simple the operation of adding an isolated vertex is, it is difficult to
accept that it may affect the multiplicities of the nonzero eigenvalues.

\textbf{Odd transversals. }The study of transversals is one the oldest and
most important topics in hypergraph theory (see, Ch. 2 of \cite{Ber87}). The
concept \textquotedblleft odd transversal\textquotedblright\ seems to have
been used first by Cowan et al. in \cite{CHKS07} and later by Rautenbach and
Szigeti in \cite{RaSz12}. The connection of odd transversals to spectral
symmetry was studied first in \cite{Nik14}. In recent literature, odd
transversal graphs have been called \textquotedblleft odd bipartite
hypergraphs.\textquotedblright\ This combination of words is contradictory,
since hypergraphs cannot be bipartite. Besides, there is no need for a new term.

\end{document}